\documentclass[11pt,english,british]{paper}
\usepackage[T1]{fontenc}
\usepackage[latin9]{inputenc}
\usepackage[letterpaper]{geometry}
\usepackage{babel}
\usepackage{array}
\usepackage{float}
\usepackage{multirow}
\usepackage{amsthm}
\usepackage{amsmath}
\usepackage{amssymb}
\usepackage[authoryear]{natbib}
\usepackage[unicode=true]
 {hyperref}

\makeatletter

\providecommand{\tabularnewline}{\\}
\floatstyle{ruled}
\newfloat{algorithm}{tbp}{loa}
\providecommand{\algorithmname}{Algorithm}
\floatname{algorithm}{\protect\algorithmname}

\numberwithin{equation}{section}
\numberwithin{figure}{section}
\theoremstyle{plain}
\newtheorem{thm}{\protect\theoremname}
  \theoremstyle{plain}
  \newtheorem{lem}[thm]{\protect\lemmaname}
  \theoremstyle{remark}
  \newtheorem{rem}[thm]{\protect\remarkname}

\usepackage{multirow}

\makeatother

  \addto\captionsbritish{\renewcommand{\lemmaname}{Lemma}}
  \addto\captionsbritish{\renewcommand{\remarkname}{Remark}}
  \addto\captionsbritish{\renewcommand{\theoremname}{Theorem}}
  \addto\captionsenglish{\renewcommand{\lemmaname}{Lemma}}
  \addto\captionsenglish{\renewcommand{\remarkname}{Remark}}
  \addto\captionsenglish{\renewcommand{\theoremname}{Theorem}}
  \providecommand{\lemmaname}{Lemma}
  \providecommand{\remarkname}{Remark}
\addto\captionsenglish{\renewcommand{\algorithmname}{Algorithm}}
\providecommand{\theoremname}{Theorem}

\begin{document}

\title{Multi-objective integer programming: An improved recursive algorithm}

\author{Melih Ozlen, Benjamin A.~Burton, Cameron A. G. MacRae}
\maketitle
\begin{abstract}
This paper introduces an improved recursive algorithm to generate
the set of all nondominated objective vectors for the Multi-Objective
Integer Programming (MOIP) problem. We significantly improve the earlier
recursive algorithm of Özlen and Azizo\u{g}lu by using the set of
already solved subproblems and their solutions to avoid solving a
large number of IPs. A numerical example is presented to explain the
workings of the algorithm, and we conduct a series of computational
experiments to show the savings that can be obtained. As our experiments
show, the improvement becomes more significant as the problems grow
larger in terms of the number of objectives.\end{abstract}

\section{Introduction\label{sec:Introduction}}

Multi-Objective Integer Programming (MOIP) has drawn the attention
of researchers in recent years, as discussed in section \ref{sec:lit}.
MOIP is seen as an extension to classical Integer Programming (IP)
which has already been used in a wide variety of decision making environments
including logistics, planning, location/allocation, scheduling, routing,
and so on. Multiple objectives enable decision makers to consider
not just a single objective but a \emph{set} of objectives simultaneously,
such as cost, profit, waste, environmental impact, risk, etc. 

In this study we deal with the noninteractive and exact solution of
the MOIP problem, in the case where there is no information available
on the form of the utility function. That is, we focus on algorithms
for generating the full set of all nondominated objective vectors.
However, in Section \ref{sec:lit} we also discuss the literature
on noninteractive and exact methods where information about the utility
function is known. We refer the reader to \citet{ehrgott2005book}
for a more detailed discussion on multi-objective optimisation, theory,
and methodology. 

The main contribution of this paper is to improve the recursive algorithm
of \citet{ozlenazizoglu2009ejor} for generating the full nondominated
set. A key drawback of the former algorithm is that it does not make
use of any information obtained from the already solved subproblems.
Our new algorithm incorporates this valuable information, and is able
to avoid solving a large set of intermediate IPs as a result. 

The remainder of this paper is organised as follows. Section \ref{sec:lit}
reviews the related literature, and in Section \ref{sec:The-problem-and}
we describe the problem and explain the original recursive algorithm.
In Section \ref{sec:Improved-recursive algorithm} we introduce our
improved algorithm, and Section \ref{sec:Numerical-example} offers
a detailed illustration of its workings using an instance of quad-objective
assignment problem (QOAP). We present the results of a computational
experiment in Section \ref{sec:Computational-experience}, and discuss
the savings that can be obtained using this new algorithm. We conclude
and provide several future research directions in Section \ref{sec:conclusion}.

\section{Literature\label{sec:lit}}

Here we survey the literature on exact noninteractive approaches to
MOIP, beginning with the special case in which an explicit utility
function is known, and then moving towards the most general case of
MOIP.

For the special case in which there is an explicit utility function,
we aim to identify a solution that optimises the given utility function.
\citet{abbaschaabane2006,jorge2009} deal with the case of a linear
utility function and propose methods to identify an optimal solution,
which must be one of the extreme supported nondominated objective
vectors. \citet{ozlenazizogluburtonJOGO} handle the case of a nonlinear
utility function, where the solution may be any member of the nondominated
set. 

In a more general case of a linear but unknown utility function, it
is sufficient to identify all extreme supported nondominated objective
vectors, since the optimal solution must come from this set. \citet{przybylskigandibleuxehrgotti2010joc,ozpeynirci20102302}
propose similar algorithms to identify all extreme nondominated objective
vectors for the MOIP problem. Both algorithms use a weighted single
objective function and partition the weight space in order to enumerate
the extreme supported nondominated set; an approach first proposed
for Multi-Objective Linear Programming (MOLP) by \citet{bensonSun2000,bensonsun2002}.

In the most general case where there is no information available about
the utility function, the aim is to generate all nondominated objective
vectors, since these can optimise an arbitrary linear or nonlinear
utility function. \citet{kleinhannan1982ejor} develop an approach
based on the sequential solutions of the single-objective models.
Their algorithm generates a subset, but not necessarily the whole
set, of all nondominated objective vectors. \citet{sylvacrema2004ejor}
improve the approach of \citet{kleinhannan1982ejor} by defining a
weighted combination of all objectives, and their approach guarantees
to generate the full nondominated set. The main drawback of their
method is that with every iteration a number of binary variables and
constraints are added to the subproblems, making it impractical to
solve problems which require a large number of iterations. \citet{laumanns_et_al:DSP:2005:246,laumannsthiele2006}
develop an adaptive version of the $\epsilon$-constraint method to
generate all nondominated objective vectors; the main handicap of
their algorithm is that it needs to solve a large number of IPs to
generate weak nondominated objective vectors as it progresses. 

\citet{przybylskigandibleuxehrgott2010do} propose a generalisation
of the two-phase algorithm for MOIP, where first extreme supported
nondominated objective vectors are identified, and then the remaining
nondominated vectors are identified using this earlier set. For the
first phase, their algorithm requires an efficient method of generating
extreme supported nondominated objective vectors for MOIP. This is
relatively easy for problems with unimodular constraint sets, such
as assignment, transportation or minimum cost network flow problems.
In these easier cases, one can use any algorithm for generating extreme
nondominated objective vectors for the MOLP problem; see \citet{burtonozlen}
for a recent discussion on this topic. In general, however, generating
all extreme supported nondominated objective vectors for MOIP is hard,
as discussed in \citet{przybylskigandibleuxehrgotti2010joc,ozpeynirci20102302}.
Likewise, the second phase of their algorithm runs well for specific
well-studied problems such as the assignment problem, but for general
MOIP it remains difficult to use and requires problem-specific implementation;
see \citet{Przybylskietal2009}.

An alternative and general approach for generating all nondominated
objective vectors for MOIP is given by \citet{ozlenazizoglu2009ejor},
whose algorithm recursively identifies objective efficiency ranges
using problems with fewer objectives. This algorithm forms the basis
for this paper, and we describe it in detail in the following section.

\section{The problem and the recursive algorithm\label{sec:The-problem-and}}

In its most general form the MOIP problem is defined as:

\smallskip{}

Min $f_{1}(x),f_{2}(x),\ldots f_{k-1}(x),f_{k}(x)$

s.t $x\in X$

where $X$ is the set of feasible points defined by $Ax=b,x_{j}\geq0$
and $x_{j}\mathbb{\in Z}$ for all $j\in\{1,2,\ldots,n\}$.

\smallskip{}

The individual objectives are defined as $f_{1}(x)=\sum_{j=1}^{n}c_{1j}x_{j}$,
$f_{2}(x)=\sum_{j=1}^{n}c_{2j}x_{j}$, \ldots, and $f_{k}(x)=\sum_{j=1}^{n}c_{kj}x_{j}$,
where $c_{ij}\in\mathbb{Z}$ for all $i\in\{1,2,\ldots,k\}$ and $j\in\{1,2,\ldots,n$\}. 

A point $x'\in X$ is called \emph{$k$-objective efficient} if and
only if there is no $x\in X$ such that $f_{i}(x)\leq f_{i}(x')$
for each $i\in\{1,2,\ldots,k\}$ and $f_{i}(x)<f_{i}(x')$ for at
least one $i$. The resulting objective vector $(f_{1}(x'),f_{2}(x'),\ldots,f_{k}(x'))$
is said to be \emph{$k$-objective nondominated}. There may be many
efficient points in the decision space that correspond to the same
nondominated objective vector, and so it can be extremely costly to
generate all efficient points. Therefore the focus of this paper (as
with most other papers of this type) is to generate the smaller set
of nondominated vectors in objective space.

\citet{ozlenazizoglu2009ejor} describe a recursive algorithm to generate
the full set of nondominated objective vectors for the MOIP problem.
A key tool in their algorithm is Constrained Lexicographic Multi Objective
Integer Programming (CLMOIP), which is formulated as:

\smallskip{}

Lexicographic objective 1: Min $f_{1}(x),f_{2}(x),\ldots f_{k-1}(x)$

Lexicographic objective 2: Min $f_{k}(x)$

s.t.\hfill$(\ast)$

$f_{k}(x)\leq l_{k}$

$x\in X$.

\smallskip{}

For any value of the bound $l_{k}$, each solution to this CLMOIP
problem yields a nondominated objective vector for our original MOIP
problem. The algorithm of \citet{ozlenazizoglu2009ejor} essentially
operates by repeatedly solving this CLMOIP problem, storing any solutions
that are found, and then shrinking the bound $l_{k}$ in order to
generate new solutions (eventually terminating when the bound $l_{k}$
is so small as to render the constraints infeasible). The recursion
arises because the $(k-1)$-objective version of this same algorithm
is used to minimise the first lexicographic objective above. Algorithm
\ref{Flo:alg-org} describes the full process in pseudocode; see \citet{ozlenazizoglu2009ejor}
for details and proofs. The set $ND_{k}$ returned by Algorithm \ref{Flo:alg-org}
contains all $k$-objective non-dominated objective vectors for the
original MOIP problem. 
\begin{lem}
If $M$ is the maximum value of $f_{k}$ amongst all CLMOIP solutions,
then after solving this CLMOIP problem we have identified all solutions
to our original MOIP problem with $M=f_{k}\leq l_{k}$ (and typically
several with $f_{k}<M$ also) and $M$ provides an upper bound on
the $f_{k}$ values of all nondominated objective vectors satisfying
the constraint $f_{k}(x)\leq l_{k}$.\end{lem}
\begin{proof}
We only summarise the proof, see \citet{ozlenazizoglu2009ejor} for
a more detailed version of this proof.

i) For any MOIP problem where we minimise $k$ objectives, any nondominated
objective vector providing an upper bound on $f_{k}$ values of all
nondominated objective vectors should be nondominated with respect
to first $k-1$ objectives otherwise this point would not be nondominated.

ii) Any MOIP problem with the additional constraint, $f_{k}(x)\leq l_{k}$,
still is a MOIP problem so i) holds, thus any solution providing an
upper bound on $f_{k}$ should be nondominated with respect to first
$k-1$ objectives for the constrained MOIP problem.

iii) Any CLMOIP problem defined above generates all nondominated objective
vectors with respect to first $k-1$ objectives and at the same time
minimises the $k^{th}$ objective to make sure all the objective vectors
generated are nondominated for $k^{th}$ objective and thus nondominated
with respect to all $k$ objectives.

iv) Following i), ii) and iii), if $M$ is the maximum value of $f_{k}$
amongst all CLMOIP solutions, there cannot be any nondominated objective
vector with $M<f_{k}\leq l_{k}$ and $M$ provides an upper bound
on the $f_{k}$ values of all nondominated objective vectors satisfying
the constraint on $f_{k}.$ 
\end{proof}
The minimum decrement of the objective function is $1$, due to integer
objective coefficients, which allows us to replace the bound $l_{k}$
with $M-1$ in subsequent CLMOIP runs, as seen in Step 2 of Algorithm
\ref{Flo:alg-org}.

\begin{algorithm}[h]
\caption{ \citet{ozlenazizoglu2009ejor}'s recursive algorithm for the original
MOIP problem}
\label{Flo:alg-org}\medskip{}

Step 0. Set $l_{k}=\infty$ and initialise $ND_{k}$ to the empty
set.

\medskip{}

Step 1. Solve the CLMOIP problem $(\ast)$, using Algorithm \ref{Flo:alg-org}
to optimise the first $k-1$ objectives. 

\hspace{1.34cm}If the problem is infeasible then,

\hspace{1.34cm}\qquad{}STOP.

\hspace{1.34cm}Let the ($k-1$)-objective nondominated set be $ND_{k-1}^{*}$

\medskip{}

Step 2. $ND_{k}=ND_{k}\cup ND_{k-1}^{*}$.

\hspace{1.34cm}$l_{k}=\max\{f_{k}\,|\, f\in ND_{k-1}^{*}\}-1$.

\hspace{1.34cm}Go to Step 1.

\medskip{}
\end{algorithm}

\begin{rem}
Algorithm \ref{Flo:alg-org} iterates by solving tighter versions
of the CLMOIP problem due to the decreasing values of $l_{k}$ that
restrict the feasible region.
\end{rem}
As can be seen in the numerical example presented in Section \ref{sec:Numerical-example},
this algorithm at the lowest level can end up solving a large number
of IPs that generate the same nondominated objective vectors again
and again. The main reason behind this is that the algorithm does
not make use of already solved subproblems or their solutions.

\section{An improved recursive algorithm\label{sec:Improved-recursive algorithm}}

A major drawback of the recursive algorithm described above is that
it does not store or utilise information on subproblems that have
already been solved. With this in mind, we propose an improved algorithm
that uses such information to avoid solving a large number of low-level
IPs. The main idea is, when solving a new CLMOIP problem, to search
for a \emph{relaxation} of this problem that has been solved before,
enabling us to skip the new CLMOIP problem entirely. We only allow
relaxation of the constraints on individual objectives, thus avoid
any issues that may arise from allowing a linear relaxation. 

The following two lemmas show how a relaxation to a CLMOIP problem
can be used to avoid solving it. Both results are straightforward,
and so we omit the proofs.
\begin{lem}
\label{lem:inf}Let $\mathcal{P}$ be a CLMOIP problem, and let $\mathcal{R}$
be a relaxation of $\mathcal{P}$. If $\mathcal{R}$ is infeasible,
then $\mathcal{P}$ is also infeasible.
\end{lem}
\vspace{0.01in}

\begin{lem}
\label{lem:fea}Let $\mathcal{P}$ be a CLMOIP problem, and let $\mathcal{R}$
be a relaxation of $\mathcal{P}$. If every nondominated objective
vector for $\mathcal{R}$ is also feasible for $\mathcal{P}$, then
the set of all nondominated objective vectors for $\mathcal{P}$ is
precisely the set of all nondominated objective vectors for $\mathcal{R}$.
\end{lem}
\vspace{0.01in}

\begin{rem}
\label{remark}If $\mathcal{R}$ has even a single nondominated objective
vector that is not feasible for $\mathcal{P}$, then the solution
to $\mathcal{R}$ cannot be used to avoid solving $\mathcal{P}$.

As we recurse down through Algorithm \ref{Flo:alg-org}, we accumulate
constraints of the form $f_{i}\leq l_{i}$. In general, each intermediate
CLMOIP problem that we solve is of the form:
\end{rem}
\smallskip{}

Lexicographic objective 1: Min $f_{1}(x),f_{2}(x),\ldots f_{q}(x)$

Lexicographic objective 2: Min $f_{q+1}(x)$

s.t.

$f_{q+1}(x)\leq l_{q+1}$, $f_{q+2}(x)\leq l_{q+2}$, $\ldots$, $f_{k}(x)\leq l_{k}$

$x\in X$.

\smallskip{}

\noindent We denote such a problem using the notation $(q,l_{q+1},l_{q+2},\ldots,l_{k})$.
It is straightforward to identify relaxations using this notation:
\begin{lem}
\label{lem:relax}The CLMOIP problem $(q,l_{q+1}',l_{q+2}',\ldots,l_{k}')$
is a relaxation of $(q,l_{q+1},l_{q+2}\ldots,l_{k})$ if $l_{i}'\geq l_{i}$
for all $i=q+1,\ldots,k$ and if $l_{i}'>l_{i}$ for some $i=q+1,\ldots,k$.\end{lem}
\begin{rem}
Since Algorithm \ref{Flo:alg-org} iterates by incrementally lowering
the bounds $l_{2},l_{3},\ldots,l_{k}$, as the algorithm progresses
it becomes highly likely that we can find a relaxation of the current
CLMOIP amongst our set of already solved problems. 
\end{rem}
\vspace{0.01in}

\begin{rem}
There could be many relaxations to a given CLMOIP, each with different
nondominated sets---all of the relaxations should be examined until
one is found that allows us to avoid solving the current CLMOIP.
\end{rem}
Using these ideas, Algorithm \ref{Flo:alg-imp} improves the earlier
Algorithm \ref{Flo:alg-org} by making use of already solved CLMOIP
problems and their solution sets.

\begin{algorithm}[h]
\caption{ Improved recursive algorithm to generate nondominated set of MOIP}
\label{Flo:alg-imp}\medskip{}

Step 0. Set $l_{k}=\infty$.

\medskip{}

Step 1. Repeat:

\hspace{1.34cm}\qquad{}Check the list of previously solved CLMOIPs
to find a relaxation to the current CLMOIP problem $(\ast)$.

\hspace{1.34cm}\qquad{}If all nondominated objective vectors for
the relaxation are feasible for the current CLMOIP then,

\hspace{1.34cm}\qquad{}\qquad{}Let that nondominated set be $ND_{k-1}^{*}$
and go to Step 3.

\hspace{1.34cm}\qquad{}If the relaxation is infeasible then,

\hspace{1.34cm}\qquad{}\qquad{}STOP

\hspace{1.34cm}Until there are no other relaxations to the current
CLMOIP.

\medskip{}

Step 2. Solve the CLMOIP problem $(\ast)$, using Algorithm \ref{Flo:alg-imp}
to optimise the first $k-1$ objectives. 

\hspace{1.34cm}If the problem is infeasible then,

\hspace{1.34cm}\qquad{}STOP.

\hspace{1.34cm}Let the ($k-1$)-objective nondominated set be $ND_{k-1}^{*}$

\medskip{}

Step 3. $ND_{k}=ND_{k}\cup ND_{k-1}^{*}$.

\hspace{1.34cm}$l_{k}=\max\{f_{k}\,|\, f\in ND_{k-1}^{*}\}-1$.

\hspace{1.34cm}Go to Step 1.

\medskip{}
\end{algorithm}

Set $ND_{k}$ returned by Algorithm \ref{Flo:alg-imp} resides all
$k$-objective non-dominated objective vectors, stated formally: 
\begin{thm}
Algorithm \ref{Flo:alg-imp} generates all nondominated objective
vectors for the original MOIP problem.\end{thm}
\begin{proof}
\citet{ozlenazizoglu2009ejor} show that Algorithm \ref{Flo:alg-org}
generates all nondominated objective vectors. Algorithm \ref{Flo:alg-imp}
only differs in Step 1, and it is clear from Lemma \ref{lem:inf}
and Lemma \ref{lem:fea} that these changes to Step~1 do not change
the subsequent results.
\end{proof}

\section{Numerical example\label{sec:Numerical-example}}

In this section we illustrate our approach on a numerical example.
We use a randomly generated 4 objective assignment problem with the
objective function coefficients listed in Table \ref{tab:Objective-function-coefficients}.
This problem has 14 nondominated objective vectors that can be identified
using Algorithm \ref{Flo:alg-org}, as presented in Table \ref{tab:Iteration-of-Algorithm}
and its continuation Table \ref{tab:Iteration-of-Algorithm cont}.
The rows of these tables show the various CLMOIP and IP problems as
they are recursively solved.

\begin{table}
\centering
\begin{tabular}{cc|c|c|c|ccc|c|c|c|ccc|c|c|c|ccc|c|c|c|c}
$c_{1}$ & $3$ & $2$ & $4$ & $4$ & $3$ & $c_{2}$ & $2$ & $1$ & $3$ & $2$ & $2$ & $c_{3}$ & $2$ & $5$ & $5$ & $4$ & $4$ & $c_{4}$ & $2$ & $5$ & $4$ & $5$ & $2$\tabularnewline
\cline{2-6} \cline{8-12} \cline{14-18} \cline{20-24} 
 & $5$ & $4$ & $3$ & $4$ & $3$ &  & $3$ & $5$ & $5$ & $1$ & $1$ &  & $4$ & $2$ & $4$ & $2$ & $5$ &  & $2$ & $4$ & $1$ & $2$ & $2$\tabularnewline
\cline{2-6} \cline{8-12} \cline{14-18} \cline{20-24} 
 & $3$ & $1$ & $5$ & $3$ & $4$ &  & $4$ & $3$ & $4$ & $1$ & $5$ &  & $4$ & $1$ & $2$ & $4$ & $3$ &  & $3$ & $3$ & $2$ & $4$ & $1$\tabularnewline
\cline{2-6} \cline{8-12} \cline{14-18} \cline{20-24} 
 & $5$ & $5$ & $4$ & $5$ & $3$ &  & $4$ & $5$ & $5$ & $2$ & $4$ &  & $4$ & $4$ & $1$ & $3$ & $1$ &  & $1$ & $3$ & $4$ & $3$ & $5$\tabularnewline
\cline{2-6} \cline{8-12} \cline{14-18} \cline{20-24} 
 & $1$ & $1$ & $1$ & $1$ & $3$ &  & $3$ & $4$ & $1$ & $5$ & $5$ &  & $4$ & $3$ & $5$ & $4$ & $1$ &  & $2$ & $2$ & $1$ & $3$ & $1$\tabularnewline
\end{tabular}

\caption{Objective function coefficients for the example 4 objective assignment
problem\label{tab:Objective-function-coefficients}}

\end{table}

The first column of these tables shows the CLMOIP problems with $q=3$
that are solved at the highest level of the recursion. For instance,
rows 1--20 follow the CLMOIP problem $(3,\infty)$; that is:

Lexicographic objective 1: Min $f_{1}(x),f_{2}(x),f_{3}(x)$

Lexicographic objective 2: Min $f_{4}(x)$

s.t.

$f_{4}(x)\leq\infty$

$x\in X$.

The second column shows the CLMOIP problems with $q=2$ that appear
at the first level of recursion. For instance, rows 6--9 follow the
problem $(2,22,\infty)$:

Lexicographic objective 1: Min $f_{1}(x),f_{2}(x)$

Lexicographic objective 2: Min $f_{3}(x)$

s.t.

$f_{3}(x)\leq22,f_{4}(x)\leq\infty$

$x\in X$.

The third column shows the CLMOIP problems with $q=1$ at the deepest
level of recursion. For instance, row 7 describes the problem $(1,18,22,\infty)$:

Lexicographic objective 1: Min $f_{1}(x)$

Lexicographic objective 2: Min $f_{2}(x)$

s.t.

$f_{2}(x)\leq18$, $f_{3}(x)\leq22$,$f_{4}(x)\leq\infty$

$x\in X$.

Each of these deepest problems (i.e., each table row) yields a new
IP that Algorithm \ref{Flo:alg-org} must solve. The resulting nondominated
objective vectors (obtained by minimising the remaining lexicographic
objectives up the recursion stack) are given in the columns labelled
$f_{1}(x),\ldots,f_{4}(x)$. Each IP that yields a \emph{new} nondominated
objective vector is marked with an asterisk ($\ast$).

The final ``relaxation'' column shows the improvements that we gain
with the new Algorithm \ref{Flo:alg-imp}: each entry in this column
lists a previously-solved subproblem that allows us to avoid solving
the current CLMOIP. For instance, the problem $(2,10,13)$ can be
avoided by using the relaxation $(2,10,\infty)$, and the problem
$(1,\infty,14,12)$ can be avoided by using the relaxation $(1,\infty,15,12)$.
Illustrating Remark \ref{remark} we cannot use $(2,\infty,\infty)$
to avoid solving $(2,\infty,13)$ although it is a previously solved
relaxation, since one of the solutions of $(2,\infty,\infty)$, namely
$(11,19,12,14)$, is not feasible for $(2,\infty,13)$. 

The results are extremely pleasing: by reusing problems that have
already been solved, Algorithm \ref{Flo:alg-imp} is able to generate
the full nondominated set by solving only $40$ IPs (shown by the
40 rows with no relaxation entry), in contrast to the 79 IPs required
by Algorithm \ref{Flo:alg-org} (corresponding to all 79 rows of the
tables).

\begin{table}
\begin{centering}
\begin{tabular}{cccccccccl}
\hline 
3-obj & 2-obj & 1-obj &  & {\footnotesize $f_{1}(x)$} & {\footnotesize $f_{2}(x)$} & {\footnotesize $f_{3}(x)$} & {\footnotesize $f_{4}(x)$} &  & {\footnotesize Relaxation}\tabularnewline
\hline 
{\footnotesize $(3,\infty)$} & {\footnotesize $(2,\infty,\infty)$} & {\footnotesize $(1,\infty,\infty,\infty)$} & {\footnotesize {*}} & {\footnotesize $11$} & {\footnotesize $19$} & {\footnotesize $12$} & {\footnotesize $14$} & {\footnotesize \enskip{}} & \tabularnewline
\hline 
 &  & {\footnotesize $(1,18,\infty,\infty)$} & {\footnotesize {*}} & {\footnotesize $12$} & {\footnotesize $11$} & {\footnotesize $11$} & {\footnotesize $13$} &  & \tabularnewline
\hline 
 &  & {\footnotesize $(1,10,\infty,\infty)$} & {\footnotesize {*}} & {\footnotesize $13$} & {\footnotesize $9$} & {\footnotesize $16$} & {\footnotesize $11$} &  & \tabularnewline
\hline 
 &  & {\footnotesize $(1,8,\infty,\infty)$} & {\footnotesize {*}} & {\footnotesize $14$} & {\footnotesize $8$} & {\footnotesize $23$} & {\footnotesize $13$} &  & \tabularnewline
\hline 
 &  & {\footnotesize $(1,7,\infty,\infty)$} & {\footnotesize \enskip{}} & {\footnotesize inf} &  &  &  &  & \tabularnewline
\hline 
 & {\footnotesize $(2,22,\infty)$} & {\footnotesize $(1,\infty,22,\infty)$} &  & {\footnotesize $11$} & {\footnotesize $19$} & {\footnotesize $12$} & {\footnotesize $14$} &  & {\footnotesize $(1,\infty,\infty,\infty)$}\tabularnewline
\hline 
 &  & {\footnotesize $(1,18,22,\infty)$} &  & {\footnotesize $12$} & {\footnotesize $11$} & {\footnotesize $11$} & {\footnotesize $13$} &  & {\footnotesize $(1,18,\infty,\infty)$}\tabularnewline
\hline 
 &  & {\footnotesize $(1,10,22,\infty)$} &  & {\footnotesize $13$} & {\footnotesize $9$} & {\footnotesize $16$} & {\footnotesize $11$} &  & {\footnotesize $(1,10,\infty,\infty)$}\tabularnewline
\hline 
 &  & {\footnotesize $(1,8,22,\infty)$} &  & {\footnotesize inf} &  &  &  &  & \tabularnewline
\hline 
 & {\footnotesize $(2,15,\infty)$} & {\footnotesize $(1,\infty,15,\infty)$} &  & {\footnotesize $11$} & {\footnotesize $19$} & {\footnotesize $12$} & {\footnotesize $14$} &  & {\footnotesize $(1,\infty,\infty,\infty)$}\tabularnewline
\hline 
 &  & {\footnotesize $(1,18,15,\infty)$} &  & {\footnotesize $12$} & {\footnotesize $11$} & {\footnotesize $11$} & {\footnotesize $13$} &  & {\footnotesize $(1,18,\infty,\infty)$}\tabularnewline
\hline 
 &  & {\footnotesize $(1,10,15,\infty)$} &  & {\footnotesize inf} &  &  &  &  & \tabularnewline
\hline 
 & \selectlanguage{english}%
{\footnotesize $(2,11,\infty)$}\selectlanguage{british}
 & {\footnotesize $(1,\infty,11,\infty)$} &  & {\footnotesize $12$} & {\footnotesize $11$} & {\footnotesize $11$} & {\footnotesize $13$} &  & \tabularnewline
\hline 
 &  & {\footnotesize $(1,10,11,\infty)$} &  & {\footnotesize inf} &  &  &  &  & {\footnotesize $(1,10,15,\infty)$}\tabularnewline
\hline 
 & \selectlanguage{english}%
{\footnotesize $(2,10,\infty)$}\selectlanguage{british}
 & {\footnotesize $(1,\infty,10,\infty)$} & {\footnotesize {*}} & {\footnotesize $15$} & {\footnotesize $16$} & {\footnotesize $7$} & {\footnotesize $12$} &  & \tabularnewline
\hline 
 &  & {\footnotesize $(1,15,10,\infty)$} & {\footnotesize {*}} & {\footnotesize $16$} & {\footnotesize $15$} & {\footnotesize $10$} & {\footnotesize $13$} &  & \tabularnewline
\hline 
 &  & {\footnotesize $(1,14,10,\infty)$} &  & {\footnotesize inf} &  &  &  &  & \tabularnewline
\hline 
 & \selectlanguage{english}%
{\footnotesize $(2,9,\infty)$}\selectlanguage{british}
 & {\footnotesize $(1,\infty,9,\infty)$} &  & {\footnotesize $15$} & {\footnotesize $16$} & {\footnotesize $7$} & {\footnotesize $12$} &  & {\footnotesize $(1,\infty,10,\infty)$}\tabularnewline
\hline 
 &  & {\footnotesize $(1,15,9,\infty)$} &  & {\footnotesize inf} &  &  &  &  & \tabularnewline
\hline 
 & \selectlanguage{english}%
{\footnotesize $(2,6,\infty)$}\selectlanguage{british}
 & {\footnotesize $(1,\infty,6,\infty)$} &  & {\footnotesize inf} &  &  &  &  & \tabularnewline
\hline 
{\footnotesize $(3,13)$} & {\footnotesize $(2,\infty,13)$} & {\footnotesize $(1,\infty,\infty,13)$} &  & {\footnotesize $12$} & {\footnotesize $11$} & {\footnotesize $11$} & {\footnotesize $13$} &  & \tabularnewline
\hline 
 &  & {\footnotesize $(1,10,\infty,13)$} &  & {\footnotesize $13$} & {\footnotesize $9$} & {\footnotesize $16$} & {\footnotesize $11$} &  & {\footnotesize $(1,10,\infty,\infty)$}\tabularnewline
\hline 
 &  & {\footnotesize $(1,8,\infty,13)$} &  & {\footnotesize $14$} & {\footnotesize $8$} & {\footnotesize $23$} & {\footnotesize $13$} &  & {\footnotesize $(1,8,\infty,\infty)$}\tabularnewline
\hline 
 &  & {\footnotesize $(1,7,\infty,13)$} &  & {\footnotesize inf} &  &  &  &  & {\footnotesize $(1,7,\infty,\infty)$}\tabularnewline
\hline 
 & {\footnotesize $(2,22,13)$} & {\footnotesize $(1,\infty,22,13)$} &  & {\footnotesize $12$} & {\footnotesize $11$} & {\footnotesize $11$} & {\footnotesize $13$} &  & {\footnotesize $(1,\infty,\infty,13)$}\tabularnewline
\hline 
 &  & {\footnotesize $(1,10,22,13)$} &  & {\footnotesize $13$} & {\footnotesize $9$} & {\footnotesize $16$} & {\footnotesize $11$} &  & {\footnotesize $(1,10,\infty,\infty)$}\tabularnewline
\hline 
 &  & {\footnotesize $(1,8,22,13)$} &  & {\footnotesize inf} &  &  &  &  & {\footnotesize $(1,8,22,\infty)$}\tabularnewline
\hline 
 & {\footnotesize $(2,15,13)$} & {\footnotesize $(1,\infty,15,13)$} &  & {\footnotesize $12$} & {\footnotesize $11$} & {\footnotesize $11$} & {\footnotesize $13$} &  & {\footnotesize $(1,\infty,\infty,13)$}\tabularnewline
\hline 
 &  & {\footnotesize $(1,10,15,13)$} &  & {\footnotesize inf} &  &  &  &  & {\footnotesize $(1,10,15,\infty)$}\tabularnewline
\hline 
 & {\footnotesize $(2,10,13)$} & {\footnotesize $(1,\infty,10,13)$} &  & {\footnotesize $15$} & {\footnotesize $16$} & {\footnotesize $7$} & {\footnotesize $12$} &  & {\footnotesize $(2,10,\infty)$}\tabularnewline
\hline 
 &  & {\footnotesize $(1,15,10,13)$} &  & {\footnotesize $16$} & {\footnotesize $15$} & {\footnotesize $10$} & {\footnotesize $13$} &  & {\footnotesize $(2,10,\infty)$}\tabularnewline
\hline 
 &  & {\footnotesize $(1,14,10,13)$} &  & {\footnotesize inf} &  &  &  &  & {\footnotesize $(2,10,\infty)$}\tabularnewline
\hline 
 & {\footnotesize $(2,9,13)$} & {\footnotesize $(1,\infty,9,13)$} &  & {\footnotesize $15$} & {\footnotesize $16$} & {\footnotesize $7$} & {\footnotesize $12$} &  & \selectlanguage{english}%
{\footnotesize $(2,9,\infty)$}\selectlanguage{british}
\tabularnewline
\hline 
 &  & {\footnotesize $(1,15,9,13)$} &  & {\footnotesize inf} &  &  &  &  & \selectlanguage{english}%
{\footnotesize $(2,9,\infty)$}\selectlanguage{british}
\tabularnewline
\hline 
 & {\footnotesize $(2,6,13)$} & {\footnotesize $(1,\infty,6,13)$} &  & {\footnotesize inf} &  &  &  &  & \selectlanguage{english}%
{\footnotesize $(2,6,\infty)$}\selectlanguage{british}
\tabularnewline
\hline 
{\footnotesize $(3,12)$} & {\footnotesize $(2,\infty,12)$} & {\footnotesize $(1,\infty,\infty,12)$} &  & {\footnotesize $13$} & {\footnotesize $9$} & {\footnotesize $16$} & {\footnotesize $11$} &  & \tabularnewline
\hline 
 &  & {\footnotesize $(1,8,\infty,12)$} &  & {\footnotesize inf} &  &  &  &  & \tabularnewline
\hline 
 & {\footnotesize $(2,15,12)$} & {\footnotesize $(1,\infty,15,12)$} &  & {\footnotesize $15$} & {\footnotesize $16$} & {\footnotesize $7$} & {\footnotesize $12$} &  & \tabularnewline
\hline 
 &  & {\footnotesize $(1,15,15,12)$} & {\footnotesize {*}} & {\footnotesize $17$} & {\footnotesize $13$} & {\footnotesize $15$} & {\footnotesize $11$} &  & \tabularnewline
\hline 
 &  & {\footnotesize $(1,12,15,12)$} &  & {\footnotesize inf} &  &  &  &  & \tabularnewline
\hline 
 & {\footnotesize $(2,14,12)$} & {\footnotesize $(1,\infty,14,12)$} &  & {\footnotesize $15$} & {\footnotesize $16$} & {\footnotesize $7$} & {\footnotesize $12$} &  & {\footnotesize $(1,\infty,15,12)$}\tabularnewline
\hline 
 &  & {\footnotesize $(1,15,14,12)$} & {\footnotesize {*}} & {\footnotesize $19$} & {\footnotesize $15$} & {\footnotesize $14$} & {\footnotesize $11$} &  & \tabularnewline
\hline 
 &  & {\footnotesize $(1,14,14,12)$} &  & {\footnotesize inf} &  &  &  &  & \tabularnewline
\hline 
 & {\footnotesize $(2,13,12)$} & {\footnotesize $(1,\infty,13,12)$} &  & {\footnotesize $15$} & {\footnotesize $16$} & {\footnotesize $7$} & {\footnotesize $12$} &  & {\footnotesize $(1,\infty,15,12)$}\tabularnewline
\hline 
 &  & {\footnotesize $(1,15,13,12)$} &  & {\footnotesize inf} &  &  &  &  & \tabularnewline
\hline 
 & {\footnotesize $(2,6,12)$} & {\footnotesize $(1,\infty,6,12)$} &  & {\footnotesize inf} &  &  &  &  & \selectlanguage{english}%
{\footnotesize $(2,6,\infty)$}\selectlanguage{british}
\tabularnewline
\hline 
\end{tabular}
\par\end{centering}

\caption{Iteration of Algorithm \ref{Flo:alg-org} and Algorithm \ref{Flo:alg-imp}
on a numerical example\label{tab:Iteration-of-Algorithm}}
\end{table}

\begin{table}
\begin{centering}
\begin{tabular}{cccccccccl}
\hline 
3-obj & 2-obj & 1-obj &  & {\footnotesize $f_{1}(x)$} & {\footnotesize $f_{2}(x)$} & {\footnotesize $f_{3}(x)$} & {\footnotesize $f_{4}(x)$} &  & {\footnotesize Relaxation}\tabularnewline
\hline 
{\footnotesize $(3,11)$} & {\footnotesize $(2,\infty,11)$} & {\footnotesize $(1,\infty,\infty,11)$} & {\footnotesize \enskip{}} & {\footnotesize $13$} & {\footnotesize $9$} & {\footnotesize $16$} & {\footnotesize $11$} & {\footnotesize \enskip{}} & {\footnotesize $(2,\infty,12)$}\tabularnewline
\hline 
 &  & {\footnotesize $(1,8,\infty,11)$} &  & {\footnotesize inf} &  &  &  &  & {\footnotesize $(2,\infty,12)$}\tabularnewline
\hline 
 & {\footnotesize $(2,15,11)$} & {\footnotesize $(1,\infty,15,11)$} & {\footnotesize {*}} & {\footnotesize $15$} & {\footnotesize $17$} & {\footnotesize $11$} & {\footnotesize $10$} &  & \tabularnewline
\hline 
 &  & {\footnotesize $(1,16,15,11)$} &  & {\footnotesize $17$} & {\footnotesize $13$} & {\footnotesize $15$} & {\footnotesize $11$} &  & \tabularnewline
\hline 
 &  & {\footnotesize $(1,12,15,11)$} &  & {\footnotesize inf} &  &  &  &  & {\footnotesize $(1,12,15,12)$}\tabularnewline
\hline 
 & {\footnotesize $(2,14,11)$} & {\footnotesize $(1,\infty,14,11)$} &  & {\footnotesize $15$} & {\footnotesize $17$} & {\footnotesize $11$} & {\footnotesize $10$} &  & {\footnotesize $(1,\infty,15,11)$}\tabularnewline
\hline 
 &  & {\footnotesize $(1,16,14,11)$} & {\footnotesize {*}} & {\footnotesize $17$} & {\footnotesize $16$} & {\footnotesize $13$} & {\footnotesize $11$} &  & \tabularnewline
\hline 
 &  & {\footnotesize $(1,15,14,11)$} &  & {\footnotesize $19$} & {\footnotesize $15$} & {\footnotesize $14$} & {\footnotesize $11$} &  & \tabularnewline
\hline 
 &  & {\footnotesize $(1,14,14,11)$} &  & {\footnotesize inf} &  &  &  &  & {\footnotesize $(1,14,14,12)$}\tabularnewline
\hline 
 & {\footnotesize $(2,13,11)$} & {\footnotesize $(1,\infty,13,11)$} &  & {\footnotesize $15$} & {\footnotesize $17$} & {\footnotesize $11$} & {\footnotesize $10$} &  & {\footnotesize $(1,\infty,15,11)$}\tabularnewline
\hline 
 &  & {\footnotesize $(1,16,13,11)$} &  & {\footnotesize $17$} & {\footnotesize $16$} & {\footnotesize $13$} & {\footnotesize $11$} &  & {\footnotesize $(1,16,15,11)$}\tabularnewline
\hline 
 &  & {\footnotesize $(1,15,13,11)$} &  & {\footnotesize inf} &  &  &  &  & {\footnotesize $(1,15,13,12)$}\tabularnewline
\hline 
 & {\footnotesize $(2,12,11)$} & {\footnotesize $(1,\infty,12,11)$} &  & {\footnotesize $15$} & {\footnotesize $17$} & {\footnotesize $11$} & {\footnotesize $10$} &  & {\footnotesize $(1,\infty,15,11)$}\tabularnewline
\hline 
 &  & {\footnotesize $(1,16,12,11)$} &  & {\footnotesize inf} &  &  &  &  & \tabularnewline
\hline 
 & {\footnotesize $(2,10,11)$} & {\footnotesize $(1,\infty,10,11)$} &  & {\footnotesize inf} &  &  &  &  & \tabularnewline
\hline 
{\footnotesize $(3,10)$} & {\footnotesize $(2,\infty,10)$} & {\footnotesize $(1,\infty,\infty,10)$} & {\footnotesize {*}} & {\footnotesize $13$} & {\footnotesize $19$} & {\footnotesize $17$} & {\footnotesize $10$} &  & \tabularnewline
\hline 
 &  & {\footnotesize $(1,18,\infty,10)$} & {\footnotesize {*}} & {\footnotesize $14$} & {\footnotesize $11$} & {\footnotesize $16$} & {\footnotesize $9$} &  & \tabularnewline
\hline 
 &  & {\footnotesize $(1,10,\infty,10)$} &  & {\footnotesize inf} &  &  &  &  & \tabularnewline
\hline 
 & {\footnotesize $(2,16,10)$} & {\footnotesize $(1,\infty,16,10)$} &  & {\footnotesize $14$} & {\footnotesize $11$} & {\footnotesize $16$} & {\footnotesize $9$} &  & \tabularnewline
\hline 
 &  & {\footnotesize $(1,10,16,10)$} &  & {\footnotesize inf} &  &  &  &  & {\footnotesize $(1,10,\infty,10)$}\tabularnewline
\hline 
 & {\footnotesize $(2,15,10)$} & {\footnotesize $(1,\infty,15,10)$} &  & {\footnotesize $15$} & {\footnotesize $17$} & {\footnotesize $11$} & {\footnotesize $10$} &  & {\footnotesize $(1,\infty,15,11)$}\tabularnewline
\hline 
 &  & {\footnotesize $(1,16,15,10)$} & {\footnotesize {*}} & {\footnotesize $18$} & {\footnotesize $15$} & {\footnotesize $15$} & {\footnotesize $9$} &  & \tabularnewline
\hline 
 &  & {\footnotesize $(1,14,15,10)$} &  & {\footnotesize inf} &  &  &  &  & \tabularnewline
\hline 
 & {\footnotesize $(2,14,10)$} & {\footnotesize $(1,\infty,14,10)$} &  & {\footnotesize $15$} & {\footnotesize $17$} & {\footnotesize $11$} & {\footnotesize $10$} &  & {\footnotesize $(1,\infty,15,11)$}\tabularnewline
\hline 
 &  & {\footnotesize $(1,16,14,10)$} &  & {\footnotesize inf} &  &  &  &  & \tabularnewline
\hline 
 & {\footnotesize $(2,10,10)$} & {\footnotesize $(1,\infty,10,10)$} &  & {\footnotesize inf} &  &  &  &  & {\footnotesize $(2,10,11)$}\tabularnewline
\hline 
{\footnotesize $(3,9)$} & {\footnotesize $(2,\infty,9)$} & {\footnotesize $(1,\infty,\infty,9)$} &  & {\footnotesize $14$} & {\footnotesize $11$} & {\footnotesize $16$} & {\footnotesize $9$} &  & \tabularnewline
\hline 
 &  & {\footnotesize $(1,10,\infty,9)$} &  & {\footnotesize inf} &  &  &  &  & {\footnotesize $(1,10,\infty,10)$}\tabularnewline
\hline 
 & {\footnotesize $(2,15,9)$} & {\footnotesize $(1,\infty,15,9)$} & {\footnotesize {*}} & {\footnotesize $16$} & {\footnotesize $18$} & {\footnotesize $15$} & {\footnotesize $9$} &  & \tabularnewline
\hline 
 &  & {\footnotesize $(1,17,15,9)$} &  & {\footnotesize $18$} & {\footnotesize $15$} & {\footnotesize $15$} & {\footnotesize $9$} &  & \tabularnewline
\hline 
 &  & {\footnotesize $(1,14,15,9)$} &  & {\footnotesize inf} &  &  &  &  & {\footnotesize $(1,14,15,10)$}\tabularnewline
\hline 
 & {\footnotesize $(2,14,9)$} & {\footnotesize $(1,\infty,14,9)$} &  & {\footnotesize inf} &  &  &  &  & \tabularnewline
\hline 
{\footnotesize $(3,8)$} & {\footnotesize $(2,\infty,8)$} & {\footnotesize $(1,\infty,\infty,8)$} &  & {\footnotesize inf} &  &  &  &  & \tabularnewline
\hline 
\end{tabular}
\par\end{centering}

\caption{Iteration of Algorithm \ref{Flo:alg-org} and Algorithm \ref{Flo:alg-imp}
on a numerical example (cont.)\label{tab:Iteration-of-Algorithm cont}}
\end{table}

\section{Computational experiment\label{sec:Computational-experience}}

\citet{przybylskigandibleuxehrgott2010do} perform computational experiments
using 3-objective assignment problems and compare CPU times of various
non-dominated set generation algorithms including \citet{sylvacrema2004ejor}
and \citet{laumannsthiele2006}. Their results identify \citet{laumannsthiele2006}
as the best general algorithm with other algorithms failing to generate
the set, for even small problem instances, within a reasonable amount
of time. As such, we compare an improved version of \citet{laumannsthiele2006}
as discussed in \citet{laumanns_et_al:DSP:2005:246} with \citet{ozlenazizoglu2009ejor}
and the improved recursive algorithm using the 3-objective 3-dimensional
Knapsack (3O3DKP) problem instances from \citet{laumannsthiele2006}
in Table \ref{tab:CPU-time 3-KP}. \citet{ozlenazizoglu2009ejor}
is significantly faster compared to \citet{laumanns_et_al:DSP:2005:246}
in solving 303DKPs and this behaviour becomes more significant with
the growing problem size. The improved algorithm introduced in this
paper is significantly faster compared to \citet{ozlenazizoglu2009ejor}
for all problem sizes while solving 3O3DKPs. 

\begin{table}
\begin{centering}
\begin{tabular}{|c|r|r|r|r|r|r|}
\hline 
\multirow{2}{*}{$n$} & \multirow{2}{*}{{\footnotesize |$\mbox{\ensuremath{ND}}$|}} & \citet{laumanns_et_al:DSP:2005:246} & \multicolumn{2}{c|}{\citet{ozlenazizoglu2009ejor}} & \multicolumn{2}{c|}{Improved Algorithm}\tabularnewline
\cline{3-7} 
 &  & CPU time (secs) & CPU time (secs) & \# IP & CPU time (secs) & \# IP\tabularnewline
\hline 
$10$ & $9$ & $0.80$ & $0.30$ & $58$ & $0.22$ & $46$\tabularnewline
\hline 
$20$ & $61$ & $20.40$ & $10.09$ & $1119$ & $3.99$ & $333$\tabularnewline
\hline 
$30$ & $195$ & $391.24$ & $117.93$ & $4705$ & $35.65$ & $1204$\tabularnewline
\hline 
$40$ & $389$ & $1046.01$ & $416.83$ & $13628$ & $84.27$ & $2357$\tabularnewline
\hline 
$50$ & $1048$ & $7081.07$ & $2044.65$ & $38683$ & $422.52$ & $6001$\tabularnewline
\hline 
$100$ & $6500$ & $403937.83$ & $82420.58$ & $207515$ & $21358.38$ & $35450$\tabularnewline
\hline 
\end{tabular}
\par\end{centering}

\caption{Comparison of \citet{laumanns_et_al:DSP:2005:246}, \citet{ozlenazizoglu2009ejor}
and improved recursive algorithm on 3O3DKP\label{tab:CPU-time 3-KP}}
\end{table}

We also perform experimentation using the 3-objective assignment problem
(3OAP) instances from \citet{przybylskigandibleuxehrgott2010do} and
generate additional objectives using a similar distribution to theirs
to test the performance of \citet{ozlenazizoglu2009ejor} and the
improved recursive algorithm on problems with more than 3 objectives
(MOAP). The results are available in Table \ref{tab:CPU-time-AP}.
The improved algorithm is faster compared to \citet{ozlenazizoglu2009ejor}
on MOAP. CPU time improvement becomes more significant with increasing
number of objectives where the improved algorithm cuts the number
of IPs solved quite effectively using the relaxations. For instance,
the new algorithm improves the CPU time by 90\% in solving the 4-objective
20 rows AP.

\begin{table}
\begin{centering}
\begin{tabular}{|c|c|r|r|r|r|r|}
\hline 
\multirow{2}{*}{$n$} & \multirow{2}{*}{$k$} & \multirow{2}{*}{{\footnotesize |$\mbox{\ensuremath{ND}}$|}} & \multicolumn{2}{c|}{\citet{ozlenazizoglu2009ejor}} & \multicolumn{2}{c|}{Improved Algorithm}\tabularnewline
\cline{4-7} 
 &  &  & CPU time (secs) & \# IPs solved & CPU time (secs) & \# IPs solved\tabularnewline
\hline 
$5x5$ & $3$ & $12$ & $0.12$ & $152$ & $0.05$ & $71$\tabularnewline
\hline 
 & $4$ & $33$ & $4.24$ & $4332$ & $0.63$ & $704$\tabularnewline
\hline 
 & $5$ & $34$ & $41.59$ & $42229$ & $2.53$ & $3524$\tabularnewline
\hline 
$10x10$ & $3$ & $221$ & $25.61$ & $4062$ & $8.61$ & $1158$\tabularnewline
\hline 
 & $4$ & $736$ & $2771.80$ & $316448$ & $214.98$ & $16268$\tabularnewline
\hline 
$15x15$ & $3$ & $483$ & $79.07$ & $8367$ & $25.25$ & $2268$\tabularnewline
\hline 
 & $4$ & $7855$ & $21075.83$ & $1405070$ & $2580.15$ & $122986$\tabularnewline
\hline 
$20x20$ & $3$ & $1942$ & $450.72$ & $29710$ & $163.66$ & $9055$\tabularnewline
\hline 
 & $4$ & $22837$ & $146813.61$ & $5568823$ & $11808.15$ & $323703$\tabularnewline
\hline 
$25x25$ & $3$ & $3750$ & $784.80$ & $33803$ & $401.14$ & $15320$\tabularnewline
\hline 
$30x30$ & $3$ & $5195$ & $1765.90$ & $55272$ & $819.73$ & $22410$\tabularnewline
\hline 
$35x35$ & $3$ & $10498$ & $4119.68$ & $96161$ & $2062.49$ & $41828$\tabularnewline
\hline 
$40x40$ & $3$ & $14733$ & $5204.42$ & $111096$ & $3066.30$ & $55935$\tabularnewline
\hline 
$45x45$ & $3$ & $23942$ & $12069.63$ & $190405$ & $6455.39$ & $91780$\tabularnewline
\hline 
$50x50$ & $3$ & $29193$ & $16464.46$ & $215528$ & $9108.42$ & $109142$\tabularnewline
\hline 
\end{tabular}
\par\end{centering}

\caption{Comparison of \citet{ozlenazizoglu2009ejor} and improved recursive
algorithm on MOAP\label{tab:CPU-time-AP}}
\end{table}

In order to see the performance of the improved algorithm on harder
MOIP problems, we experiment using multi-objective travelling salesman
(MOTSP) instances from \citet{ozpeynirci20102302}. We also generate
some problems with higher number of objectives using their problem
generator, the results are available in Table \ref{tab:CPU-time-TSP}.
The usage of relaxations bring great CPU time improvements for the
MOTSP as well, and we see over 95\% cut in the CPU time for the 4-objective
10-city TSP problem.

\begin{table}
\centering{}%
\begin{tabular}{|c|c|c|r|r|r|r|}
\hline 
\multirow{2}{*}{$n$} & \multirow{2}{*}{$k$} & {\footnotesize |$\mbox{\ensuremath{ND}}$|} & \multicolumn{2}{c|}{\citet{ozlenazizoglu2009ejor}} & \multicolumn{2}{c|}{Improved Algorithm}\tabularnewline
\cline{3-7} 
 &  &  & CPU time (secs) & \# IPs solved & CPU time (secs) & \# IPs solved\tabularnewline
\hline 
$5$ & $3$ & $8$ & $2.31$ & $75$ & $1.51$ & $45$\tabularnewline
\hline 
 & $4$ & $10$ & $8.00$ & $453$ & $3.13$ & $141$\tabularnewline
\hline 
$10$ & $3$ & $94$ & $1006.27$ & $3410$ & $183.15$ & $605$\tabularnewline
\hline 
 & $4$ & $561$ & $281300.50$ & $1052981$ & $11984.49$ & $39665$\tabularnewline
\hline 
$15$ & $3$ & $560$ & $32841.38$ & $41702$ & $3557.41$ & $3662$\tabularnewline
\hline 
\end{tabular}\caption{Comparison of \citet{ozlenazizoglu2009ejor} and improved recursive
algorithm on MOTSP\label{tab:CPU-time-TSP}}
\end{table}

All computations are carried out on a single core of an Intel Core
i7-2600 processor on a machine with 8 GB of RAM with hyper-threading
and turbo-boasting disabled. GCC 4.7 is used to compile the code with
default options, and CPLEX 12.5 with a single thread and default settings
is used as the integer programming solver. A general C implementation
of the improved recursive algorithm to solve problems, input in an
extended LP file format, with arbitrary number of objectives is available
at \href{https://bitbucket.org/melihozlen/moip_aira/}{https://bitbucket.org/melihozlen/moip\_{}aira/}. 

The large and diverse set of problems in terms of their difficulty
that we use to compare the original and improved algorithms show that
the new algorithm brings significant improvements in terms of the
CPU time. The improvements become even larger when the number of objectives
increases.

\section{Conclusion\label{sec:conclusion}}

We propose a significant improvement on the recursive algorithm developed
by \citet{ozlenazizoglu2009ejor}, based on the systematic reuse of
solutions to relaxations of intermediate CLMOIP problems. Our numerical
example and computational experiments show that the CPU time required
by this new algorithm is significantly smaller than the requirement
of the original algorithm, and moreover this improvement becomes more
pronounced with increasing number of objectives. 

One important feature of the improved recursive algorithm (which it
inherits from the original) is that it may be used to generate only
a subset of the nondominated objective vectors. This feature is important
in cases where there are known restrictions on the individual objective
function values, or where an explicitly known utility function is
provided by the decision maker. 

One area for future research is the development of domain-specific
approaches to Multi Objective Combinatorial Optimisation (MOCO) problems
that make similar improvements to avoid solving a large number of
subproblems.

Another area might be to develop algorithms to solve MOIP which can
take advantage of the highly parallel computing resources as most
of the existing algorithms are limited to using a single core or thread.

\section*{Acknowledgements}

We are thankful to anonymous reviewers for their constructive comments
that helped us improve this paper substantially. Dr. Marco Laumanns
kindly sent us the code and problem instances of \citet{laumanns_et_al:DSP:2005:246}.
Dr. {\"O}zgur {\"O}zpeynirci kindly sent us their problem generator
and instances from \citet{ozpeynirci20102302}. Dr. Anthony Przybylski
kindly sent us their problem instances from \citet{przybylskigandibleuxehrgott2010do}.
The second author is supported by the Australian Research Council
under the Discovery Projects funding scheme (project DP1094516).

\bibliographystyle{plainnat}

\vspace{2cm}

\noindent Melih Ozlen\\
School of Mathematical and Geospatial Sciences, RMIT University\\
GPO Box 2476V, Melbourne VIC 3001, Australia\\
(melih.ozlen@rmit.edu.au)

\bigskip{}

\noindent Benjamin A.~Burton\\
School of Mathematics and Physics, The University of Queensland\\
Brisbane QLD 4072, Australia\\
(bab@maths.uq.edu.au)

\bigskip{}

\noindent Cameron A. G. MacRae

\noindent School of Mathematical and Geospatial Sciences, RMIT University\\
GPO Box 2476V, Melbourne VIC 3001, Australia\\
(cameron.macrae@rmit.edu.au)
\end{document}